\documentclass[a4paper,
               11pt,
               pdftex,
               normalheadings,
               headsepline,
               footsepline,
               onecolumn,
               headinclude,
               footinclude,
               DIV14,
               abstracton]
{scrartcl}

\usepackage
{
    graphicx,
    amssymb,
    amsmath,
    amsthm,
    xcolor,
    dsfont,
    algpseudocode,
    authblk,
}

\usepackage[ngerman, english]{babel}

\usepackage[
left=25mm,
right=25mm,
top=20mm,
bottom=35mm
]{geometry}

\usepackage[bf]{caption}
\usepackage[colorlinks,
            pdffitwindow=false,
            plainpages=false,
            pdfpagelabels=true,
            pdfpagemode=UseOutlines,
            pdfpagelayout=SinglePage,
            bookmarks=false,
            colorlinks=true,
            hyperfootnotes=false,
            linkcolor=blue,
            citecolor=green!50!black]
{hyperref}

\usepackage{enumerate}
\usepackage{algorithm}

\usepackage{tabularx}

\usepackage{hyphenat}
\DeclareMathAlphabet{\mathpzc}{OT1}{pzc}{m}{it}

\newcommand{\R}{\mathbb{R}}
\newcommand{\N}{\mathbb{N}}
\newcommand{\C}{\mathbb{C}}

\newcommand{\Fcal}{\mathcal{F}}
\newcommand{\Kcal}{\mathcal{K}}
\newcommand{\Ycal}{\mathcal{Y}}

\newcommand\xqed[1]{\leavevmode\unskip\penalty9999 \hbox{}\nobreak\hfill \quad\hbox{#1}}
\newcommand{\exampleSymbol}{\xqed{$\triangle$}}
\newcommand{\qedOwn}{\xqed{$\square$}}

\newtheorem{theorem}{Theorem}[section]

\newtheorem{remark}[theorem]{Remark}
\newtheorem{example}[theorem]{Example}
\newtheorem{assumption}[theorem]{Assumption}

\makeatletter
\renewcommand*\env@matrix[1][*\c@MaxMatrixCols c]{%
  \hskip -\arraycolsep
  \let\@ifnextchar\new@ifnextchar
  \array{#1}}
\makeatother

\newcolumntype{Y}{>{\centering\arraybackslash}X}

\let\OLDthebibliography\thebibliography
\renewcommand\thebibliography[1]{
	\OLDthebibliography{#1}
	\setlength{\parskip}{0pt}
	\setlength{\itemsep}{0pt plus 0.3ex}
}

\date{}

\begin{document}

\title{Koopman operator-based model reduction for switched-system control of PDEs}
\author[1]{Sebastian Peitz}
\author[2]{Stefan Klus}
\affil[1]{\normalsize Department of Mathematics, Paderborn University, Germany}
\affil[2]{\normalsize Department of Mathematics and Computer Science, Freie Universit{\"a}t Berlin, Germany}

\maketitle

\begin{abstract}
We present a new framework for optimal and feedback control of PDEs using Koopman operator-based reduced order models (K-ROMs). The Koopman operator is a linear but infinite-dimensional operator which describes the dynamics of observables. A numerical approximation of the Koopman operator therefore yields a linear system for the observation of an autonomous dynamical system. In our approach, by introducing a finite number of constant controls, the dynamic control system is transformed into a set of autonomous systems and the corresponding optimal control problem into a switching time optimization problem. This allows us to replace each of these systems by a K-ROM which can be solved orders of magnitude faster. By this approach, a nonlinear infinite-dimensional control problem is transformed into a low-dimensional linear problem. In situations where the Koopman operator can be computed exactly using Extended Dynamic Mode Decomposition (EDMD), the proposed approach yields optimal control inputs. Furthermore, a recent convergence result for EDMD suggests that the approach can be applied to more complex dynamics as well. To illustrate the results, we consider the 1D Burgers equation and the 2D Navier--Stokes equations. The numerical experiments show remarkable performance concerning both solution times and accuracy.
\end{abstract}

\section{Introduction}
\label{sec:Introduction}

The increasing complexity of technical systems presents a great challenge for control. We often want to control system dynamics described by partial differential equations (PDEs). If the system is nonlinear -- such as the Navier--Stokes equations for fluid flow --, this is particularly challenging. As a result, advanced control techniques such as \emph{Model Predictive Control (MPC)} \cite{GP17} have gained more and more attention in recent years. In MPC, a system model is used to repeatedly compute an open-loop optimal control on a finite-time horizon which results in a closed-loop control behavior. This requires solving the open-loop problem in a very short time, which is in general infeasible for nonlinear PDEs when using a standard discretization approach such as finite element or finite volume methods.

To overcome this problem, reduced order modeling is frequently applied to replace the high fidelity model by a surrogate model that can be solved much faster, see~\cite{BGW15} for an overview. Various methods exist for deriving such a surrogate model, the most common for nonlinear systems probably being \emph{Galerkin projection} in combination with \emph{Proper Orthogonal Decomposition (POD)} \cite{Sir87}. Many researchers have dedicated their work to developing optimal control methods based on POD for which convergence towards the true optimum can be proved, either using the singular values associated with the POD modes \cite{Row05,HV05,TV09} or by \emph{trust-region} approaches \cite{Fah00}. An approach to closed-loop flow control using POD-based surrogate models has been developed in \cite{PHA15}. A well-known drawback is that the number of required POD modes grows rapidly with increasing complexity of the system dynamics so that Galerkin models can become infeasible. Often additional measures like calibration \cite{CEMF09} or the construction of special modes \cite{NPM05} have to be taken. 

An alternative approach to construct a \emph{reduced order model (ROM)} is by means of the \emph{Koopman operator}~\cite{Koo31}, which is a linear but infinite-dimensional operator describing the dynamics of observables. This approach is particularly suited to be applied to sensor measurements, also in situations where the underlying system dynamics is unknown. A lot of work has been invested both to study the properties of the Koopman operator \cite{BMM12,Mez13} as well as to efficiently compute numerical approximations via \emph{Dynamic Mode Decomposition (DMD)} \cite{Sch10,RMB+09,TRL+14} or \emph{Extended Dynamic Mode Decomposition (EDMD)} \cite{WKR15,KKS16,KGPS18}. More recently, various attempts have been made to use ROMs based on the Koopman operator for control problems \cite{PBK18,KM16,KKB17}. In these approaches, the Koopman operator is approximated from an augmented state -- consisting of the actual state and the control -- in order to deal with the non-autonomous control system. Moreover, in most cases the entire state is used as the observable (\emph{full-state observable}) in order to reproduce the system dynamics.

In this article, we propose an alternative approach for solving both open- and closed-loop control problems using \emph{Koopman operator-based ROMs (K-ROMs)}, which, one the one hand, allows us to drastically reduce the dimension of the model and, on the other hand, gives access to present (and future) convergence results for numerical approximations of the Koopman operator. The key idea is to transform the dynamical control system (i.e., a nonlinear ODE or PDE) into a switched dynamical system by restricting the control to a finite set of constant values, where the state can be influenced by switching between different autonomous systems at each time step. By numerically approximating a Koopman operator for each of the different autonomous dynamical systems individually, the switched dynamics can be reproduced by switching between the respective Koopman operators. This results in a switching time problem for the open-loop case, whereas the closed-loop approach is realized via MPC. Using recent convergence results for the Koopman operator \cite{KM17}, identity of the full and the K-ROM based objective function is achieved. Furthermore, we make use of the fact that in a control problem, we do not necessarily require knowledge of the entire system state but only of some observation which is used in the controller. This allows us to drastically reduce the dimension of the optimization problem -- instead of using a high-dimensional finite-element discretization and a higher order time integrator, we can use a matrix-vector product to predict the dynamics of the low-dimensional observation. 
For related concepts, the reader is referred to\cite{VMS10} and \cite{GL99,LHPT08,CDH16} for approaches based on the Perron--Frobenius operator and \emph{occupation measures}, respectively.

The remainder of the article is structured as follows: In Section~\ref{sec:Theory}, we introduce the basic concepts for the Koopman operator and its numerical approximation as well as the relevant control techniques. Our new approach is then introduced in Section~\ref{sec:Control_Koopman} and results are shown in Section \ref{sec:Results} for an ODE problem, for the 1D Burgers equation, and for the 2D Navier--Stokes equations. We conclude with a short summary and possible future work in Section~\ref{sec:Conclusion}.

\section{Preliminaries}
\label{sec:Theory}

The goal of this article is to significantly accelerate the solution of a PDE-constrained optimal control problem of the form:
\begin{equation}\label{eq:OCP}
	\begin{aligned}
		\min_{u \in \mathcal{U}} J(y) &= \min_{u \in \mathcal{U}} \int_{t_0}^{t_e}L(y(\cdot,t)) \ dt \\
		\mbox{s.t.} \quad \dot{y}(\cdot,t) &= G(y(\cdot,t),u(t)), \\
		y(\cdot,0) &= y^0,
	\end{aligned}
\end{equation}
where $y$ is the system state and $y(\cdot, t)$ is an element of an appropriate function space $\Ycal$ (e.g., the Sobolev space $H^s(\Omega, \R^{n_y})$ with $\Omega$ being the domain, $n_y$ the spatial dimension and $s \geq 1$ the required differentiability). Furthermore, $u \in L^2([t_0,t_e], U^{n_u})$ is the $n_u$-dimensional control with box constraints $U = [u^l, u^u]$, and $G \colon \Ycal \times U \rightarrow \Ycal$ describes the system dynamics. For ease of notation, the objective function $L \colon \Ycal\rightarrow\R$ only depends on $y$ explicitly. However, the framework presented here can be extended to objectives also depending on the control $u$ in a straightforward manner. In order to achieve the desired acceleration of \eqref{eq:OCP}, we are taking two steps:
\begin{enumerate}[i)]
	\item replace $G$ by a finite number of autonomous systems
		\[G_{\overline{u}}(y(\cdot,t)) = G(y(\cdot,t),\overline{u})\]
		with constant input $\overline{u} \in \hat{U} = \{u^0,\ldots,u^{n_c-1}\}$;
	\item construct linear systems for low-dimensional observations of the infinite-dimensional systems $G_{u^j}$ using the Koopman operator.
\end{enumerate}

In the following Subsections~\ref{subsubsec:STMPC_STO} and~\ref{subsubsec:STMPC_MPC}, \emph{Switching Time Optimization (STO)} and \emph{Model Predictive Control (MPC)} are introduced. We will rely on STO for open-loop and on MPC for closed-loop control of the switching problem introduced by step i).
Then, in Section~\ref{subsec:Koopman}, the Koopman operator and its numerical approximation via EDMD are introduced which we will utilize for constructing the reduced system introduced in step ii). Convergence of EDMD towards the Koopman operator is discussed since we will rely on this result for our convergence statement. 

\subsection{Switching Time Optimization and MPC}
\label{subsec:STMPC}

In this section we introduce the concepts of STO and MPC which we will later use in open- and closed-loop control algorithms.

\subsubsection{Switching Time Optimization}
\label{subsubsec:STMPC_STO}
Switched systems are very common in engineering. They can be seen as a special case of hybrid systems which possess both continuous and discrete-time control inputs (cf.~\cite{Lib03,ZA15} for an introduction and a survey). The switched systems we want to consider here are characterized by $n_c$ different (autonomous) right-hand sides $G_{u^0}, \ldots, G_{u^{n_c-1}}$. We introduce the switching sequence $\tau \in \R^{p+2}$ with $\tau_0 = t_0$ and $\tau_{p+1} = t_e$. The entries $\tau_1,\ldots,\tau_p$ (with $\tau_l \geq \tau_{l-1}$) describe the time instants at which the right-hand side of the dynamical system is changed:
\begin{equation}\label{eq:switched_dynamics}
	\begin{aligned}
		\dot{y}(\cdot,t) &= G_{u^j}(y(\cdot,t)) \quad \text{for}~t\in [\tau_{l-1}, \tau_l), \\
		y(\cdot,0) &= y^0, \\
		j &= l~\mbox{mod}~n_c.
	\end{aligned}
\end{equation}
The third line in \eqref{eq:switched_dynamics} indicates that (following \cite{EWD03}) the switching sequence of the system is predetermined, i.e., we switch from $G_{u^0}$ to $G_{u^1}$, from $G_{u^1}$ to $G_{u^2}$ and so on. 
Finally, we go back from $G_{u^{n_c-1}}$ to $G_{u^0}$.

Problem \eqref{eq:OCP} can be transformed into a switched system by restricting $u$ to $\hat{U} = \{u^0,\ldots,u^{n_c-1}\} \subset U$, i.e., $G_{u^j}(y(\cdot,t)) = G(y(\cdot,t),u^j)$. For simplicity, we always place the $u^j$ equidistantly between $u^0 = u^l$ and $u^{n_c-1} = u^u$. However, it may be advantageous to determine problem-specific values for the $u^j$ in the offline phase.
Using this transformation, problem \eqref{eq:OCP} can be written in terms of the switching instants:
\begin{equation}\label{eq:STO}
	\begin{aligned}
		\min_{\tau \in \R^{p+2}} J(y) &= \min_{\tau \in \R^{p+2}} \int_{t_0}^{t_e}L(y(\cdot,t)) \ dt\\
		\mbox{s.t.} \quad &\eqref{eq:switched_dynamics}.
	\end{aligned}
\end{equation}
Switched systems appear in many applications. Consider, for instance, a chemical reactor where a valve is either open or closed. Moreover, continuous control inputs may be approximated by a finite number of fixed control inputs as we do here. Motivated by this, switching time problems have been extensively studied in the literature, see, e.g., \cite{Lib03}. Consequently, there exist efficient methods to determine the optimal switching sequence using gradient-based and even second-order, Newton-type methods which rely on a reformulation in terms of the switching times \cite{EWD03,SOBG17}. Alternative methods are based on relaxation and penalization or sum-up rounding \cite{Sag09,SBD12}.

\begin{example}\label{ex:ODE_PBK16}
	Consider the following ODE example \cite{PBK18}:
	\begin{equation*}
	\begin{aligned}
	\dot{y}(t) &= G(y(t),u(t)) = \left(\begin{array}{c}
	\alpha \, y_1(t) \\ \beta (y_2(t) - (y_1(t))^2) + u(t)
	\end{array}\right),\\
	y(0) &= y^0.
	\end{aligned}
	\end{equation*}
	By restricting $u$ to $n_c$ constant values, we transform the control system into $n_c$ autonomous systems:
	\begin{equation}\label{eq:BBPK16}
	\begin{aligned}
	\dot{y}(t) &= G_{u^j}(y(t)) = \left(\begin{array}{c}
	\alpha \, y_1(t) \\ \beta (y_2(t) - (y_1(t))^2)
	\end{array}\right) + \left(\begin{array}{c}
	0 \\ u^j
	\end{array}\right),\\
	y(0) &= y^0.
	\end{aligned}
	\end{equation}
	The system dynamics of \eqref{eq:BBPK16} (with $\alpha = -0.05$ and $\beta = -1$) are visualized in Figure~\ref{fig:SwitchingTime_BBPK16} for $n_c = 3$ and a fixed switching sequence with 10 intervals. \exampleSymbol
	\begin{figure}[h]
		\centering
		\parbox[b]{0.49\textwidth}{\centering \includegraphics[width=.4\textwidth]{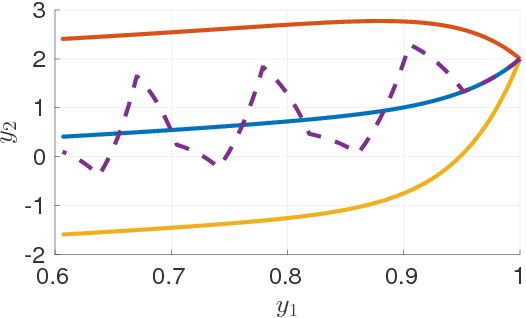} \\ \footnotesize{(a)}}
		\parbox[b]{0.49\textwidth}{\centering \includegraphics[width=.4\textwidth]{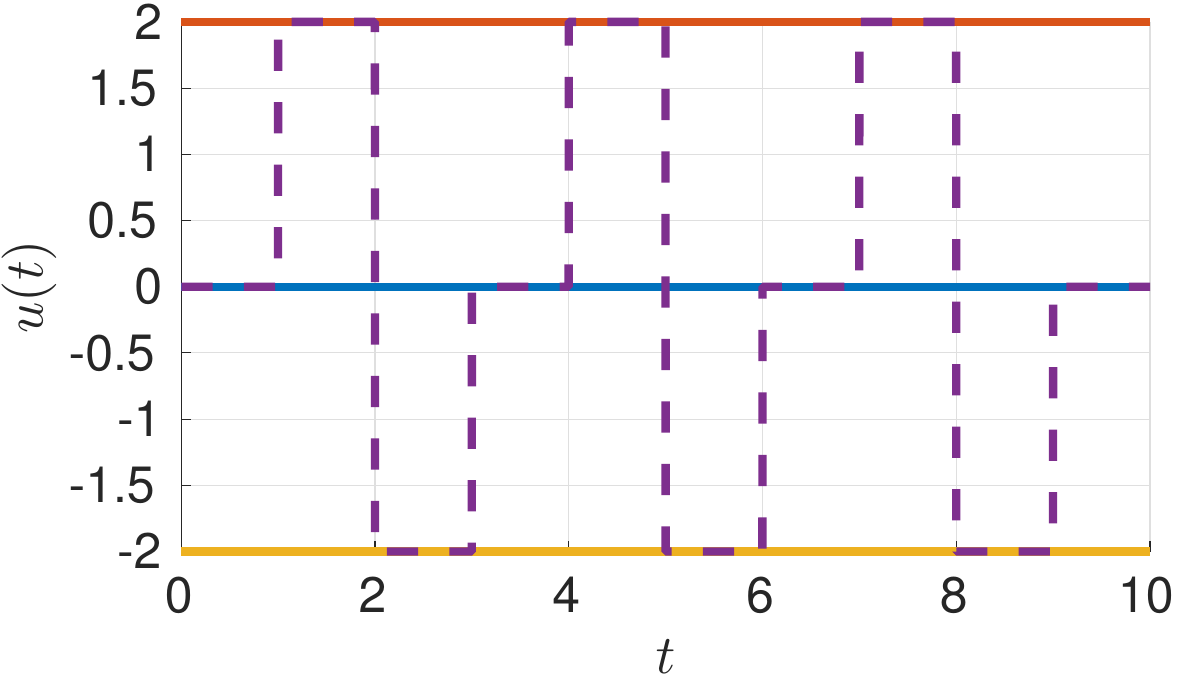} \\ \footnotesize{(b)}}
		\caption{(a) The trajectories of three different autonomous systems ($u^0 = 0$, $u^1 = 2$, $u^2 = -2$) with starting point $ y^0 = (1, 2)^\top $ and the trajectory of the switched system (dashed line) with switching according to (b).}
		\label{fig:SwitchingTime_BBPK16}
	\end{figure}
\end{example}

\subsubsection{Model Predictive Control}
\label{subsubsec:STMPC_MPC}
For real systems, it is often insufficient to determine a control input a priori. Due to the so-called \emph{plant-model mismatch} -- the difference between the dynamics of the real system and the model -- the open-loop control input will not be able to control the system as desired or at least be non-optimal. Furthermore, disturbances cannot be taken into account by open-loop control strategies. A remedy to this issue is MPC \cite{GP17} (cf.~also \cite{IK02} for the infinite-dimensional case), where open-loop problems are solved repeatedly on finite horizons. 
In order to use the discrete-time formulation from \cite{GP17}, we introduce the flow map $\Phi : \Ycal \times U \rightarrow \Ycal$ of the continuous system $G$ with a fixed time step $h$ (obtained by a numerical discretization scheme, for instance), and the control $u$ is constant over each interval $h$. For this system, an open-loop problem is solved in real-time over a so-called \emph{prediction horizon} of length $p$:
\begin{equation}\label{eq:MPC}
	\begin{aligned}
		\min_{u \in U^p} &\sum_{i=s}^{s+p-1} L(y_i) \\
		\text{s.t.}\quad y_{i+1} &= \Phi(y_i, u_{i-s+1}), \\
		y_0 &= y^0,
	\end{aligned}
\end{equation}
where we have introduced the abbreviation $y_i = y(\cdot,t_i)$. The first part of this solution is then applied to the real system while the optimization is repeated with the prediction horizon moving forward by one sample time. (The indexing $i-s+1$ is required to account for the finite-horizon control and the infinite-horizon state.) For this reason, MPC is also referred to as \emph{moving horizon control} or \emph{receding horizon control}.


Similar to the transformation of the STO problem, we now replace the dynamical control system by $n_c$ autonomous systems $\Phi_{u^0}$ to $\Phi_{u^{n_c-1}}$ and thereby transform Problem \eqref{eq:MPC} to a switching problem:
\begin{equation}\label{eq:MPC_STO}
	\begin{aligned} 
		\min_{u \in \hat{U}^p} &\sum_{i=s}^{s+p-1} L(y_i) \\
		\mbox{s.t.}\quad y_{i+1} &= \Phi_{u_{i-s+1}}(y_i), \\
		y_0 &= y^0.
	\end{aligned}
\end{equation}
In other words, each entry of $u$ describes which system $\Phi_{u_i}$ to apply in the $i^{\mathsf{th}}$ step. Due to the discrete-time dynamics, Problem~\eqref{eq:MPC_STO} is now a combinatorial problem that can be solved using dynamic programming, for instance. 

\subsection{Koopman Operator and EDMD}
\label{subsec:Koopman}

Let $ \Phi \colon \Ycal \to \Ycal $ be a discrete deterministic dynamical system defined on the state space $\Ycal$ and let $ f \colon \Ycal \rightarrow \R^q $ be a real-valued observable of the system. Then the Koopman operator $ \Kcal \colon \Fcal \to \Fcal $ with $\Fcal = L^{2}(\Ycal)$ (see~\cite{Mez13,WKR15}), which describes the evolution of the observable $ f $, is defined by
\begin{equation*}
	(\Kcal f)(y) = f(\Phi(y)).
\end{equation*}
The Koopman operator is linear but infinite-dimensional. Its adjoint, the Perron--Frobenius operator, describes the evolution of densities. The definition of the Koopman operator can be naturally extended to continuous-time dynamical systems as described in \cite{LaMa94,BMM12}. Given an autonomous system of the form 
\begin{equation*}
	\dot{y}(\cdot,t) = G(y(\cdot,t)),
\end{equation*}
the \emph{Koopman semigroup} of operators $ \{ \Kcal^t \} $ is defined as
\begin{equation*}
	(\Kcal^t f)(y) = f(\Phi^t(y)),
\end{equation*}
where $ \Phi^t $ is the flow map associated with $ G $. In what follows, we will consider discrete dynamical systems given by the discretization of ODEs or PDEs. That is, $ \Phi = \Phi^h $ for a fixed time step $ h $.

One method to compute a numerical approximation of the Koopman operator from data is EDMD \cite{WKR15,KKS16}. The following brief description is based on the review paper \cite{KNKWKSN18}. EDMD is a generalization of DMD \cite{Sch10,TRL+14} and can be used to compute a finite-dimensional approximation of the Koopman operator, its eigenvalues, eigenfunctions, and modes. In contrast to DMD, EDMD allows arbitrary basis functions -- which could be, for instance, monomials, Hermite polynomials, or trigonometric functions -- for the approximation of the dynamics. We will sometimes not be able to observe the full (potentially infinite-dimensional) state of the system, but consider only a finite number of measurements, given by $ z = f(y) \in \R^q $. For a given set of basis functions $ \{ \psi_{1},\,\psi_{2},\,\dots,\,\psi_{k} \} $, we then define a vector-valued function $ \psi \colon \R^q \to \R^k $ by
\begin{equation*}
	\psi(z) =
	\begin{bmatrix}
	\psi_{1}(z) & \psi_{2}(z) & \dots & \psi_{k}(z)
	\end{bmatrix}^{\top}.
\end{equation*}
If $ \psi(z) = z $, we obtain DMD as a special case of EDMD. We assume that we have either measurement or simulation data, written in matrix form as
\begin{equation*}
	Z =
	\begin{bmatrix}
	z_1 & z_2 & \cdots & z_m
	\end{bmatrix}
	\quad\text{and}\quad
	\widetilde{Z} =
	\begin{bmatrix}
	\widetilde{z}_1 & \widetilde{z}_2 & \cdots & \widetilde{z}_m
	\end{bmatrix},
\end{equation*}
where $ \widetilde{z}_i = f(\Phi(y_i)) $. The data could either be obtained via many short simulations or experiments with different initial conditions or one long-term trajectory or measurement. If the data is extracted from one long trajectory, then $ \widetilde{z}_i = z_{i+1} $. The data matrices are embedded into the typically higher-dimensional feature space by
\begin{align*}
	\Psi_{Z} = \begin{bmatrix} \psi(z_{1}) & \dots & \psi(z_{m}) \end{bmatrix} \mbox{ and } \Psi_{\widetilde{Z}} =
	\begin{bmatrix} \psi(\widetilde{z}_{1}) & \dots & \psi(\widetilde{z}_{m}) \end{bmatrix}.
\end{align*}
With these data matrices, we then compute the matrix $ K \in \R^{k \times k} $ defined by
\begin{equation*}
K^{\top} = \Psi_{\widetilde{Z}} \Psi_Z^+ = \big( \Psi_{\widetilde{Z}} \Psi_Z^{\top} \big) \big(\Psi_Z \Psi_Z^{\top}\big)^+,
\end{equation*}
where $^+$ denotes the pseudoinverse. 
The matrix $ K $ can be viewed as a finite-dimensional approximation of the Koopman operator. 
\begin{remark}
	The decomposition of the Koopman operator into modes, eigenvalues and eigenfunctions is commonly used to analyze the system dynamics as well as predict the future state. In the situation we are presenting here, we can pursue an even simpler approach and obtain the discrete-time update for the observable $z$ directly using $K$:
	\begin{equation*}
		\psi(z_{i+1}) \approx K^{\top} \psi(z_{i}), \quad i = 0,1,\ldots
	\end{equation*}
\end{remark}
First results showing convergence of EDMD towards the Koopman operator have recently been proven in~\cite{KM17}. In short, the result states that as both the basis size $k$ as well as the number of measurements $m$ tend to infinity, the matrix $K^{\top} (=\Kcal_{k,m})$ converges to the Koopman operator. As we will utilize this result for our convergence analysis, it is repeated below.
Before stating the two theorems required for the convergence, we introduce the following two assumptions.
\begin{assumption}\label{ass:muIndependence}
	The basis functions $\psi_{1}, \ldots ,\psi_k$ are such that
	\begin{align*}
	\mu \{z \in \mathcal{Z} \mid c^{\top}\psi(z) = 0 \} = 0
	\end{align*}
	for all $c \in \R^k$, $c \neq 0$, where $\mu$ is a given probability distribution according to which the data samples $z_1,\ldots,z_m$ are drawn and $ \mathcal{Z} \subset \R^q $ is the space of all measurements.
\end{assumption}
\noindent This assumption ensures that the measure $\mu$ is not supported on a zero level set of a linear combination of the basis functions used, cf.~\cite{KM17} for details.
\begin{assumption} \label{ass:KoopmanBoundedness}
	The following conditions hold:
	\begin{enumerate}
		\item The Koopman operator $\Kcal \colon \Fcal \to \Fcal$ is bounded.
		\item The observables $\psi_{1}, \ldots ,\psi_k$ defining $\Fcal_k$ (i.e., the finite-dimensional representation of $\Fcal$) are selected from a given orthonormal basis of $\Fcal$, i.e., $(\psi_i)_{i=1}^{\infty}$ is an orthonormal basis of $\Fcal$.
	\end{enumerate}
\end{assumption}

The convergence of $\Kcal_{k,m}$ to $\Kcal$ is now achieved in two steps. In the first step, convergence of $\Kcal_{k,m}$ to $\Kcal_{k}$ is shown as the number of samples $m$ tends to infinity. Here, $\Kcal_k$ is the projection of $\Kcal$ onto $\Fcal_k$. The second step then yields convergence of $\Kcal_{k}$ to $\Kcal$ as the basis size $k$ increases.
\begin{theorem}[\cite{KM17}] \label{thm:Convergence_Koopman1}
	If Assumption~\ref{ass:muIndependence} holds, then we have with probability one for all $\phi \in \Fcal_k$ 
	\begin{equation*}
	\lim_{m\rightarrow\infty} \|\Kcal_{k,m} \phi - \Kcal_k \phi \| = 0,
	\end{equation*}
	where $\| \cdot \|$ is any norm on $\Fcal_k$. In particular, we obtain
	\begin{equation*}
	\lim_{m\rightarrow\infty} \|\Kcal_{k,m} - \Kcal_k \| = 0,
	\end{equation*}
	where $\| \cdot \|$ is any operator norm and 
	\begin{equation*}
	\lim_{m\rightarrow\infty} \mbox{dist}(\sigma(\Kcal_{k,m}), \sigma(\Kcal_{k})) = 0,
	\end{equation*}
	where $\sigma(\cdot) \subset \C$ denotes the spectrum of an operator and $\mbox{dist}(\cdot, \cdot)$ the Hausdorff metric on subsets of $\C$.
\end{theorem}
\begin{theorem}[\cite{KM17}] \label{thm:Convergence_Koopman}
	Let Assumption~\ref{ass:KoopmanBoundedness} hold and define the $L_2(\mu)$ projection of a function $\phi$ onto $\Fcal_k$ by
	\begin{equation*}
	P_k^{\mu} \phi = \arg \min_{f\in \Fcal_k} \| f - \phi \|_{L_2(\mu)}.
	\end{equation*}
	Then, as $k \rightarrow \infty$, the sequence of operators $\Kcal_k P_k^{\mu} = P_k^{\mu} \Kcal P_k^{\mu}$ converges strongly to $\Kcal$ in $L^2(\mu)$. 
\end{theorem}

\section{Open- and closed-loop control using K-ROMs}
\label{sec:Control_Koopman}

If the system dynamics $G$ are known, then the techniques from Section~\ref{subsec:STMPC} can immediately be applied. However, if the underlying system dynamics are described by a PDE, then solving the problem numerically (e.g., with a finite element method) can quickly become very expensive such that real-time applicability is not feasible. Furthermore, there are many systems where the dynamics are not known explicitly. In both situations, we can use observations to approximate the Koopman operator and derive a linear system describing the dynamics of these observations. These could consist of (part of) the system state as well as arbitrary functions of the state such as the lift coefficient of an object within a flow field.

We want to use such a Koopman operator-based reduced order model (K-ROM) for both open- and closed-loop problems. Similar to reduced-basis approaches, this introduces a splitting into an \emph{offline phase} and an \emph{online phase}. In the offline phase, we collect data and compute reduced models. In the online phase, these models are then used to accelerate the optimization problems.
Several approaches for Koopman operator-based control have recently been proposed, see e.g., \cite{PBK18,KKB17} for open-loop problems and \cite{KM16} for closed-loop problems, where the authors also use MPC. All these approaches have in common that one single Koopman operator is computed for an augmented state $(y,u)$. This requires collecting data from a large number of state-control combinations and it can become difficult to represent these rich dynamics with one single Koopman operator.

We here propose an alternative approach where we compute $n_c$ Koopman operators for the $n_c$ different autonomous systems that have been introduced in Section~\ref{subsec:STMPC}:
\begin{equation*}
	(\Kcal_{u^j} f)(y) = f(\Phi_{u^j}(y)),\quad j = 0,\ldots,n_c-1.
\end{equation*}
Using EDMD, we can compute approximations of the individual Koopman operators and then define discrete linear systems
\begin{equation}\label{eq:Discrete_Koopman_Dynamics}
	\eta_{i+1} = {K_{u^j}}^{\top} \eta_{i},\quad j = 0,\ldots,n_c-1,
\end{equation}
with initial condition $\eta_0 = \psi(z_0) = \psi(f(y^0))$, which denotes the observation of the system state at time $t_0$, expressed in terms of the dictionary $\psi$. These linear dynamics now replace the original differential equation, and
due to the linearity of the model and the restriction to low-dimensional observables instead of the full state $y$, we can significantly accelerate the computation. It should be noted that \eqref{eq:Discrete_Koopman_Dynamics} does not imply $\psi(z_{i+1}) = {K_{u^j}}^{\top} \psi(z_{i})$ for all $i$ but only $\psi(z_{i+1}) \approx {K_{u^j}}^{\top} \psi(z_{i})$ since the $K_{u^j}$ are computed from an overdetermined system.

\subsection{Switching Time Optimization}
For the switching time optimization, the problem formulation has to be adapted since we are now restricted to the time step $h$ of the flow map $\Phi$ (i.e., the sample time between two consecutive snapshots of the sampled data). Consequently, the switched dynamics~\eqref{eq:switched_dynamics} are replaced by a discrete version:
\begin{equation} \label{eq:switched_dynamics_discrete}
	\begin{aligned}
		\eta_{i+1} &= K_{u^j}^{\top} \eta_{i} \quad \text{for}~i=\hat{\tau}_{l-1} h,\ldots, \hat{\tau}_l h,\\
		\eta_0 &= \psi(f(y^0)), \\
		j &= l~\mbox{mod}~n_c,
	\end{aligned}
\end{equation}
and the switching problem~\eqref{eq:STO} is replaced by an integer version:
\begin{equation}\label{eq:STO_Koopman}
	\begin{aligned}
		\min_{\hat{\tau} \in \N^{p+2}} J &= \min_{\hat{\tau} \in \N^{p+2}} \sum_{i=0}^{k - 1} \hat{L}(\eta_i)\\
		\mbox{s.t.} \quad &\eqref{eq:switched_dynamics_discrete},
	\end{aligned}
\end{equation}
where $\hat{L}$ is the reduced objective function formulated with respect to the observations. We again have $p$ switching instants (with $\hat{\tau}_l \geq \hat{\tau}_{l-1}$, $l=1,\ldots,p+1$) and $\hat{\tau}_0 = t_0 / h$ and $\hat{\tau}_{p+2} = t_e / h$.

We now compare the two problem formulations~\eqref{eq:STO} and \eqref{eq:STO_Koopman}, i.e., the switching time problem and the corresponding approximation using the K-ROM. First, we assume that the full objective function $L$ can be evaluated using only observations:
\begin{assumption}\label{ass:equality_objectives_STO}
	$L(y(\cdot,t)) = \hat{L}(\eta_i)$ for all $\psi$ and for all $t\in\{t_0,t_0+h,t_0+2h,\ldots,t_e\}$ and $i = (t-t_0)/h$.
\end{assumption} 
\noindent Note that in a practical setting, this is automatically satisfied since the objective function can only be evaluated using observations (e.g., sensor data) such that the objective $L$ has to be defined accordingly. To circumvent projection issues from $\eta$ to $z$, a straightforward approach is to include the identity as basis functions in $\psi$.

\begin{remark}\label{rem:Optimality_STO}
	Assumption~\ref{ass:equality_objectives_STO} is not sufficient for the solutions to be close since in Problem~\eqref{eq:STO_Koopman}, we are restricted to the time grid defined by the sample time $h$ which is potentially much larger than the numerical discretization of Problem \eqref{eq:STO}. Furthermore, the objective function of Problem \eqref{eq:STO} is defined for continuous time. Therefore, identity of the objective function values can only be shown when performing two (significant) changes: 
	\begin{itemize}
		\item replace the objective function by one similar to \eqref{eq:STO_Koopman}, i.e., we only evaluate $L$ at a finite number of time steps $\{0, h, \ldots, (k-1)h\}$,
		\item restrict the optimal switching times to this time grid as well
			\begin{equation}\label{eq:tau_integer}
				\frac{\tau_l}{h} \in \N \quad \mbox{for} \quad l = 1,\ldots,p.
			\end{equation}
	\end{itemize}
	The requirement \eqref{eq:tau_integer} obviously has an impact on the solution. We will see in an example in Section~\ref{sec:Results} that if it is omitted, the solutions are not identical, but remain close. Using the two changes above, a result could be derived similar to Theorem~\ref{thm:Convergence_MPC} for the MPC case. 
\end{remark}

\subsection{Model Predictive Control}
In the closed-loop setting, the full problem and the K-ROM approximation are related more closely due to the discrete formulation of the MPC problem~\eqref{eq:MPC_STO} such that we do not have to restrict the solution to a subset of the feasible set of the original problem. The reduced version is obtained by replacing the objective function as well as the system dynamics by the K-ROM formulations:
\begin{equation}\label{eq:MPC_STO_Koopman}
	\begin{aligned} 
		\min_{u \in \hat{U}^p} &\sum_{i=s}^{s+p-1} \hat{L}(\eta_i)  \\
		\mbox{s.t.}\quad \eta_{i+1} &= K_{u_{i-s+1}}^{\top} \eta_{i}, \\
		\eta_0 &= \psi(f(y^0)).
	\end{aligned}
\end{equation}
In this situation, the two objective functions possess the same value for all $u \in U^p$ and almost every initial condition $\eta_0$.
\begin{theorem}\label{thm:Convergence_MPC}
	Consider Problem \eqref{eq:MPC_STO} and the corresponding approximation \eqref{eq:MPC_STO_Koopman} using the K-ROM and let Assumptions \ref{ass:muIndependence}, \ref{ass:KoopmanBoundedness}, and \ref{ass:equality_objectives_STO} be satisfied. Then, as the basis size and number of sampled data points tend to infinity, the objective functions $L$ and $\hat{L}$ corresponding to Problems \eqref{eq:MPC_STO} and \eqref{eq:MPC_STO_Koopman} are identical for every $u \in \hat{U}^p$ and $z$ almost everywhere.
\end{theorem}
\noindent
\textbf{Proof.}
	Provided that the Assumptions~\ref{ass:muIndependence} and \ref{ass:KoopmanBoundedness} are satisfied, Theorem~\ref{thm:Convergence_Koopman} implies convergence of EDMD in measure. That is, for all $\epsilon > 0$ we obtain
	\begin{equation*}
		\begin{aligned}
			&\lim_{k,m\rightarrow\infty} \mu \left( \left\{ z_i \in \mathcal{Z} \mid \|K_{u_i}^{\top} \psi(z_i) - \mathcal{K}_{u_i} \psi(z_i) \| \geq \epsilon \right\} \right) \\ 
			=&\lim_{k,m\rightarrow\infty} \mu \left( \left\{ z_i \in \mathcal{Z} \mid \|K_{u_i}^{\top} \psi(z_i) - \psi(f(\Phi_{u_i}(y(\cdot,t_{i})))) \| \geq \epsilon \right\} \right) \\
            =& ~~ 0.
		\end{aligned}
	\end{equation*}
	This then extends to entire trajectories and we have for $i=1,\ldots,p$:
	\begin{equation*}
		\begin{aligned}
			\lim_{k,m\rightarrow\infty} \mu \left( \left\{ z_0 \in \mathcal{Z} \mid \left\| \hat{\varphi}(\eta_0, t_i, u) - \psi(f(\varphi(y^0, t_i, u))) \right\| \geq \epsilon \right\} \right) = 0,
		\end{aligned}
	\end{equation*}
	where $\varphi: \Ycal \times \R \times \hat{U}^i \rightarrow \Ycal$ and $\hat{\varphi}: \R^k \times \R \times \hat{U}^i \rightarrow \R^k$ denote the flows of the full and the reduced dynamics, respectively.
	By Assumption~\ref{ass:equality_objectives_STO}, identity of the objective functions of Problems \eqref{eq:MPC_STO} and \eqref{eq:MPC_STO_Koopman} follows for all $u\in\hat{U}^p$ and $z$ almost everywhere. \qedOwn
\begin{remark}\label{rem:Assumptions}
	Note that Assumptions \ref{ass:muIndependence} and \ref{ass:KoopmanBoundedness} (in particular the boundedness of $\mathcal{K}$) do not hold for all systems or may be hard to verify, especially for nonlinear PDEs. Nevertheless, the convergence results for the Koopman operator justify the choice of EDMD as the numerical tool. 
	For problems where stronger convergence properties can be shown (i.e., pointwise convergence of $K_u$), stronger statements can likely be made for Problem~\eqref{eq:MPC_STO_Koopman}.
\end{remark}

The MPC procedure now follows classical approaches as discussed in Section~\ref{subsubsec:STMPC_MPC}. It is summarized in Algorithm~\ref{alg:MPC_Koopman}. In order to achieve real-time applicability, Problem~\eqref{eq:MPC_STO_Koopman} (Step 4) has to be solved within the sample time~$h$.
Since the aim of the article is to introduce the K-ROM-based switched systems concept, we simply evaluate the objective function value for all possible $u$ and select the optimal solution. However, for large prediction horizons (i.e., large $p$), this is not feasible anymore. Instead, dynamic programming techniques can be used or, since these also suffer from the \emph{curse of dimensionality}, relaxation approaches as proposed in \cite{Sag09}.
Note that the Koopman operator can also be used to predict the initial condition of the next optimization problem (Step 3), i.e., it serves as a state estimator. For state prediction, accuracy plays an important role and due to the convergence result of the Koopman operator, excellent prediction accuracy can be guaranteed.

\begin{algorithm}[h]
	\caption{(K-ROM-based MPC)}
	\label{alg:MPC_Koopman}
	\begin{algorithmic}[1]
		\Require EDMD approximations of $n_c$ Koopman operators; prediction horizon length $p \in \N$.
		\For{$i=0, 1, 2, \ldots$}.
		\State \parbox[t]{\dimexpr\linewidth-\algorithmicindent}{Observe current state: $\eta_i = \psi(z_i) = psi(f(y_i))$. \strut}
		\State \parbox[t]{\dimexpr\linewidth-\algorithmicindent}{Predict $\eta_{i+1}$ using \eqref{eq:Discrete_Koopman_Dynamics}. \strut}
		\State \parbox[t]{\dimexpr\linewidth-\algorithmicindent}{Solve Problem~\eqref{eq:MPC_STO_Koopman} with initial condition $\eta_{i+1}$ on the prediction horizon of length $p$. \strut}
		\State \parbox[t]{\dimexpr\linewidth-\algorithmicindent}{At $t = (i+1)\,h$, apply the first entry of the solution, i.e., $\hat{\tau}^*_1$, to the system. \strut}
		\EndFor
	\end{algorithmic}
\end{algorithm}

\section{Results}
\label{sec:Results}

We now illustrate the results using several examples of varying complexity. We will first revisit the ODE problem~\eqref{eq:BBPK16} for the switching time optimization. In the MPC framework, we will then consider control of the 1D Burgers equation and of the incompressible 2D Navier--Stokes equations. For the latter problems, the K-ROM predictions become too inaccurate after several steps such that open-loop control cannot be realized.

The details of the numerical setup, the data sampling and the EDMD approximations of the respective Koopman operators $\mathcal{K}_{u^j}$ are summarized in Table~\ref{tab:Sampling}. In the first case, data is sampled individually for the respective autonomous dynamics, whereas in the other two cases, long-term simulations ($60$ and $900$ seconds, respectively)  with a random switching sequence are split according to the active system for the individual snapshots. Numerical experiments show that a relatively small number of data points for each system ($< 100$) is sufficient. We use a set of basis functions comprising monomials of order up to 2 or 3, respectively. However, choosing problem-dependent basis functions might further improve the accuracy of the K-ROMs.

\begin{table}[t]
	\centering
	\caption{Numerical setup and efficiency for different examples.}
	\begin{tabular}{lccc}
		\hline
		Problem & ODE \eqref{eq:BBPK16} & 1D Burgers & 2D NSE \\ 
		\hline 
		Integrator & RK4 & Expl. Euler & PISO \cite{FP02} \\ 
		Integr.~time step & $0.005$ & $0.005$ & $0.01$ \\ 
		Time step $h$ & $0.04$ & $0.5$ & $0.25$ \\ 
		Monom.~order & 2 & 3 & 2 \\ 
		dim($\Ycal_{\mathsf{(discretized)}}$) & 2 & 48 & 22,000 \\ 
		$q$ (dim($z$)) & 2 & 4 & 8 \\ 
		$k$ (dim($K$)) & 6 & 35 & 45 \\ 
		Integr.~speed-up & $\approx 20$ & $\approx 100$ & $\approx 7.5 \cdot 10^4$ \\ 
		\hline 
	\end{tabular} 
	\label{tab:Sampling}
\end{table}

\begin{figure}[t]
	\centering
	\parbox[b]{0.49\textwidth}{\centering \includegraphics[width=0.4\textwidth]{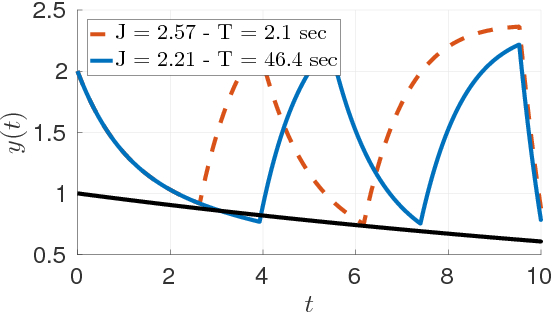} \\ \footnotesize{(a)}}
	\parbox[b]{0.49\textwidth}{\centering \includegraphics[width=0.4\textwidth]{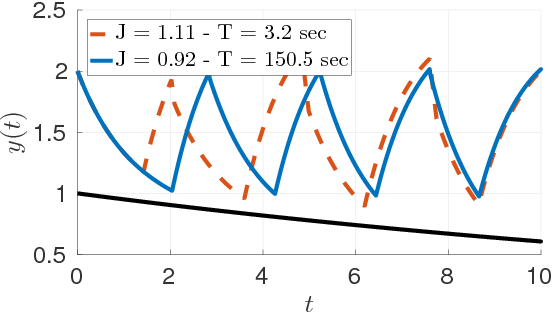} \\ \footnotesize{(b)}} \\[1ex]
	\parbox[b]{0.49\textwidth}{\centering \includegraphics[width=0.4\textwidth]{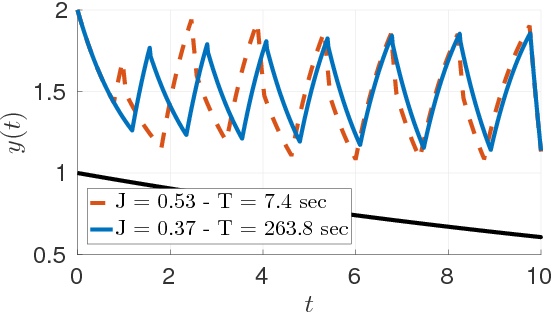} \\ \footnotesize{(c)}}
	\parbox[b]{0.49\textwidth}{\centering \includegraphics[width=0.4\textwidth]{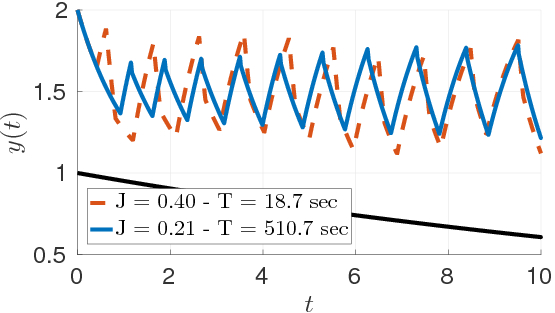} \\ \footnotesize{(d)}} \\[1ex]
	\parbox[b]{0.49\textwidth}{\centering \includegraphics[width=0.4\textwidth]{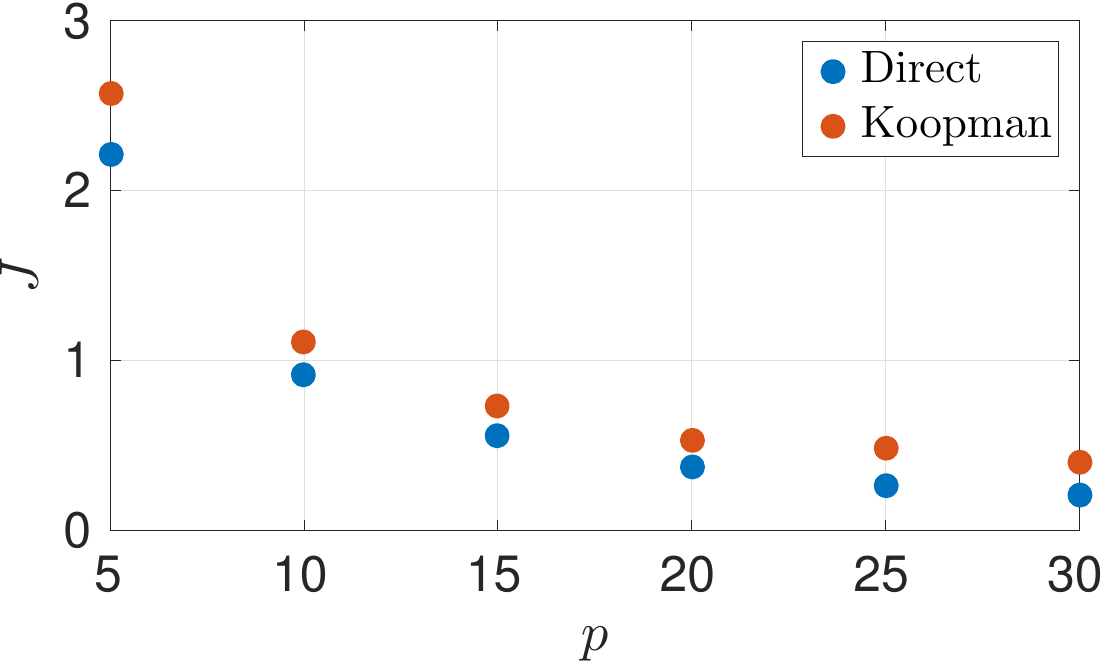} \\ \footnotesize{(e)}}
	\parbox[b]{0.49\textwidth}{\centering \includegraphics[width=0.4\textwidth]{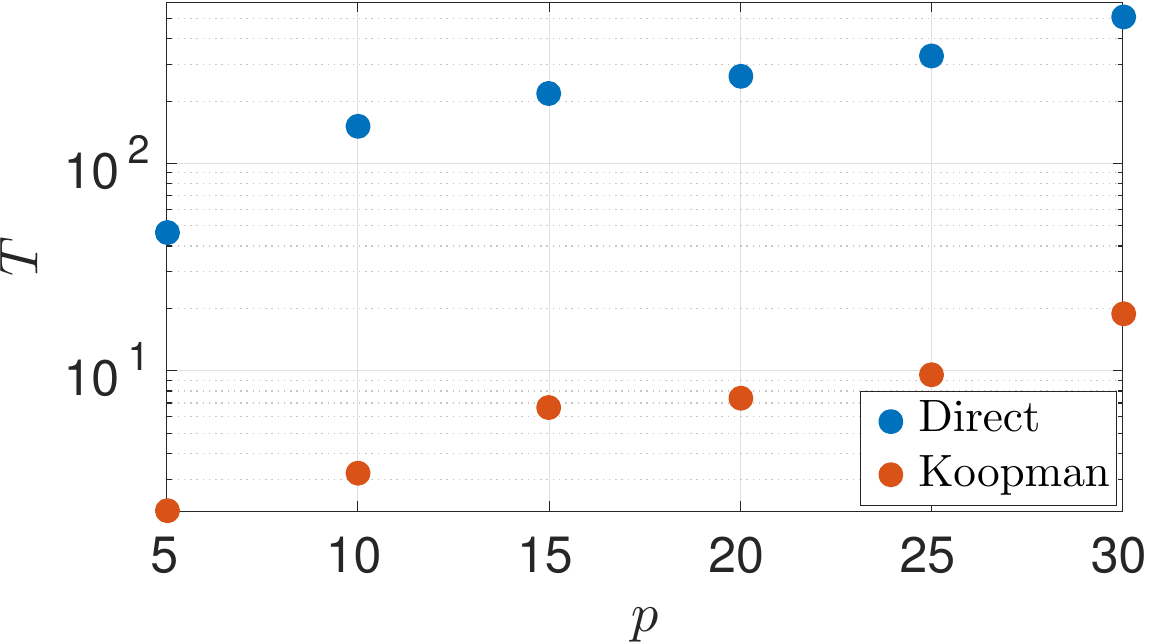} \\ \footnotesize{(f)}}
	\caption{(a) to (d) Optimal trajectories of $y_2$ of the ODE (blue solid line) and the K-ROM (orange dashed line) for different numbers of switching instants ($p=5, 10, 20, 30$). The variable $y_1$ is shown in black. (e) Comparison of the objective function values depending on the number of switching instants $p$. (f) The computing times corresponding to (e).}
	\label{fig:STO_BBPK16}
\end{figure}

Table~\ref{tab:Sampling} also shows the speed-up achieved by the K-ROM. We use MATLAB for all computations except the Navier--Stokes solution, which is computed using the open source code OpenFOAM \cite{JJT07}. The comparison is with regard to time integration over a fixed interval. We observe an acceleration by several orders, especially for systems that are expensive to evaluate. The reason is that the K-ROM is linear and its size only depends on the dimension $q$ of the observable $z$ and the size $k$ of the basis $\psi$. It is completely independent of the numerical discretization of the domain. Furthermore, the time steps of the K-ROM can be much larger than the step size for the numerical solution of the PDE.

\subsection{Switching Time Optimization}
For the switching time optimization, we revisit Problem~\eqref{eq:BBPK16} from Example~\ref{ex:ODE_PBK16} and compare the performance of the full control problem \eqref{eq:STO} with the K-ROM approximation \eqref{eq:STO_Koopman} for a tracking type objective, i.e., we want the system state to follow a prescribed trajectory $y^{\mathsf{opt}}(t)$:
\begin{equation*}
	\begin{aligned}
		J = \sum_{i=0}^{k - 1} (y_{i,2} - y_i^{\mathsf{opt}})^2.
	\end{aligned}
\end{equation*}
For the numerical solution of the switching time problems, we use a gradient-based approach as proposed in~\cite{EWD03}, but we compute the gradients using a finite difference approximation on the discretized time grid.

We omit the requirement $\tau_l / h \in \N$. Hence, we do not observe identity of the two solutions, cf.~Remark~\ref{rem:Optimality_STO}. However, when enforcing this constraint, the solutions coincide even though the objective functions are not identical, i.e., continuous versus finite sum.
The results without enforcing $\tau_l / h \in \N$ are shown in Figure~\ref{fig:STO_BBPK16}. We observe that as we increase the number of switches, the distance in $y_2$ between the full problem and the K-ROM approximation decreases. More importantly, we observe a significant speed-up (by a factor of approximately 50), cf.~Figure~\ref{fig:STO_BBPK16}~(f). This is due to the linearity of the K-ROM as well as an increased step length by a factor of $8$.

\subsection{Model Predictive Control}

\textbf{1D Burgers equation.}
As the first example with PDE constraints, we consider the 1D Burgers equation:
\begin{align*}
	\dot{y}(x,t) - \nu \Delta y(x,t) + y(x,t) \nabla y(x,t) &= u^j(x), \\
	y(x,0) &= y^0(x),
\end{align*}
with periodic boundary conditions and $\nu = 0.01$. We have directly transformed the system with a distributed control $u(x,t)$ into three autonomous systems with time-independent shape functions $u^j$, see Figure~\ref{fig:Burgers}~(a).
\begin{figure}[h]
	\centering
	\parbox[b]{0.49\textwidth}{\centering \includegraphics[width=.4\textwidth]{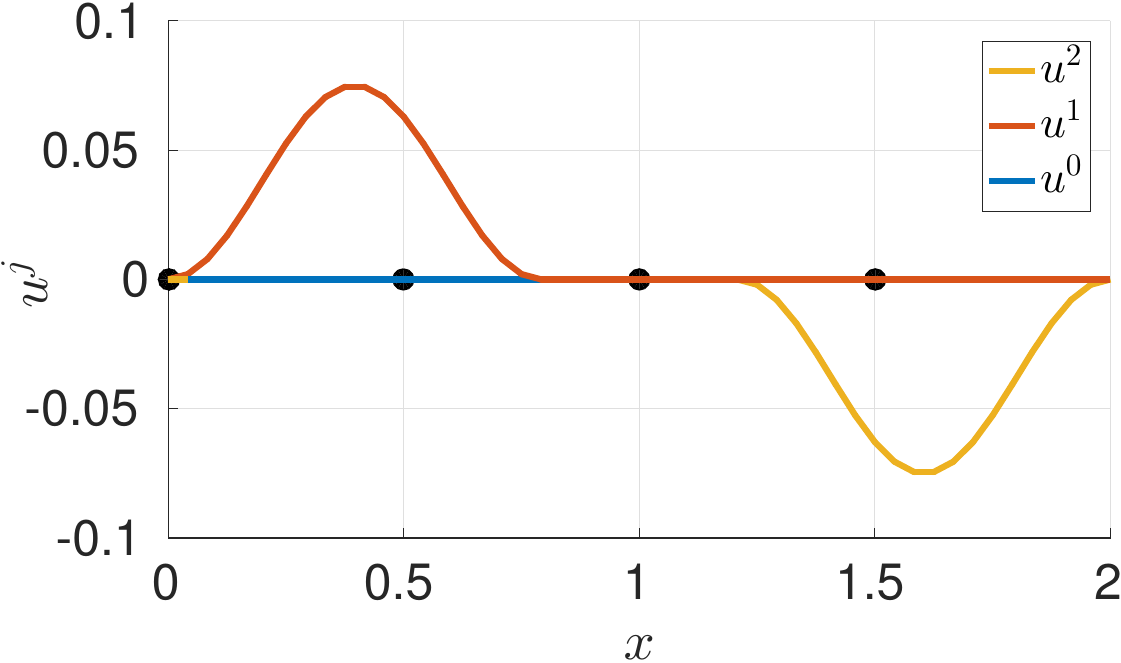} \\ \footnotesize{(a)}}
	\parbox[b]{0.49\textwidth}{\centering \includegraphics[width=.4\textwidth]{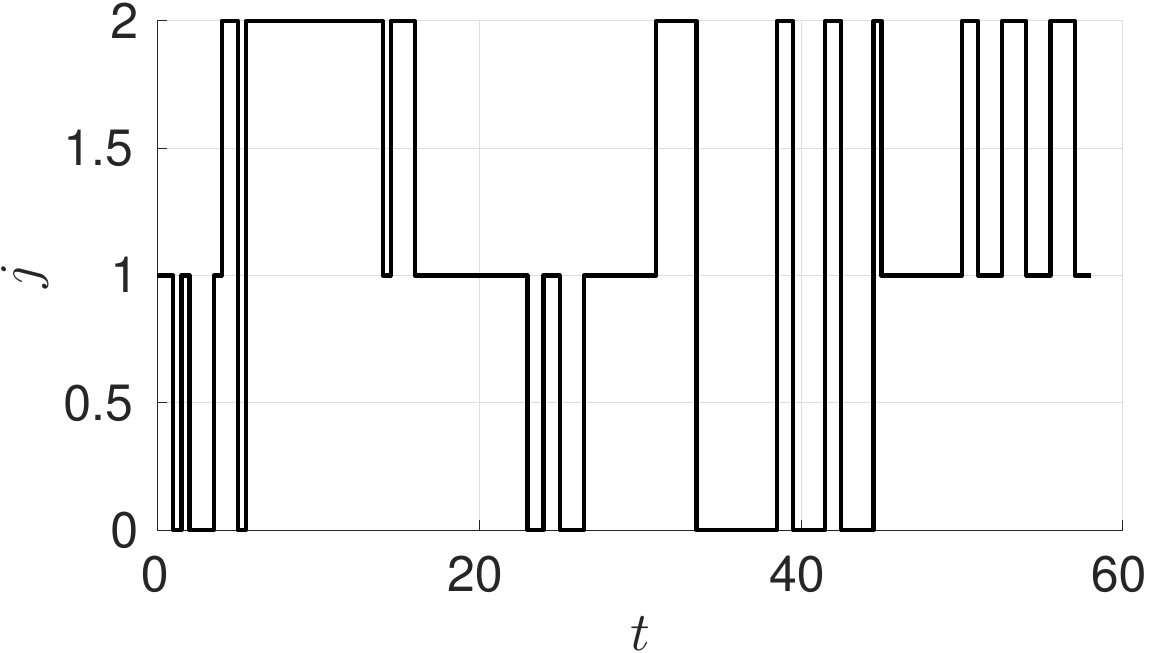} \\ \footnotesize{(b)}} \\[1ex]
	\parbox[b]{0.49\textwidth}{\centering \includegraphics[width=.4\textwidth]{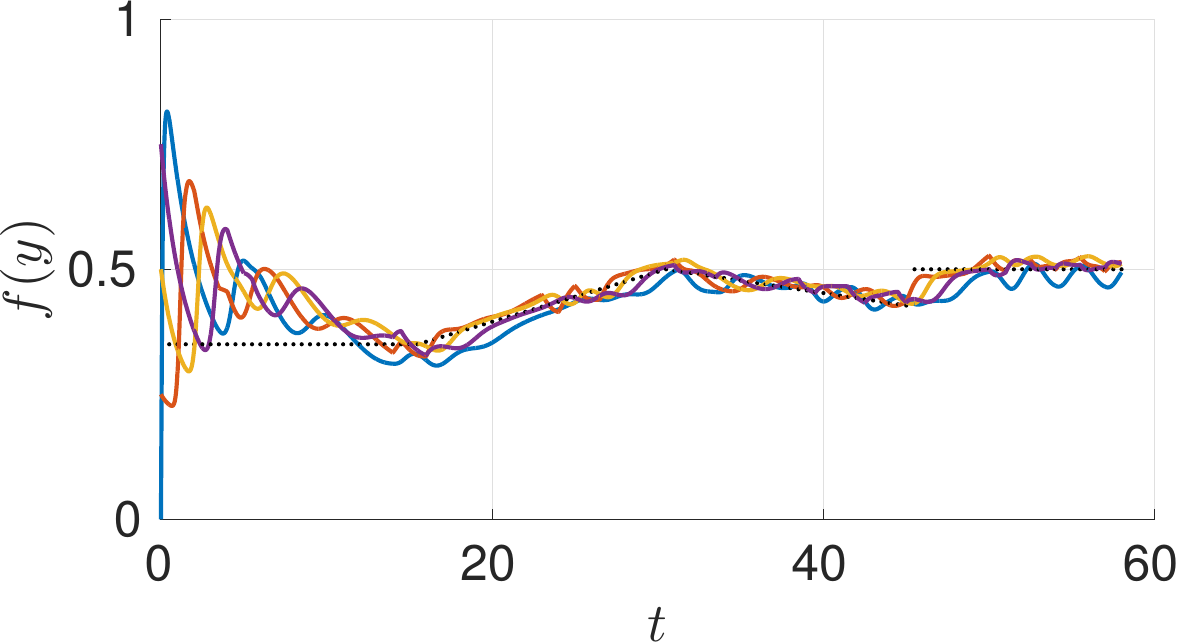} \\ \footnotesize{(c)}}
	\parbox[b]{0.49\textwidth}{\centering \includegraphics[width=.4\textwidth]{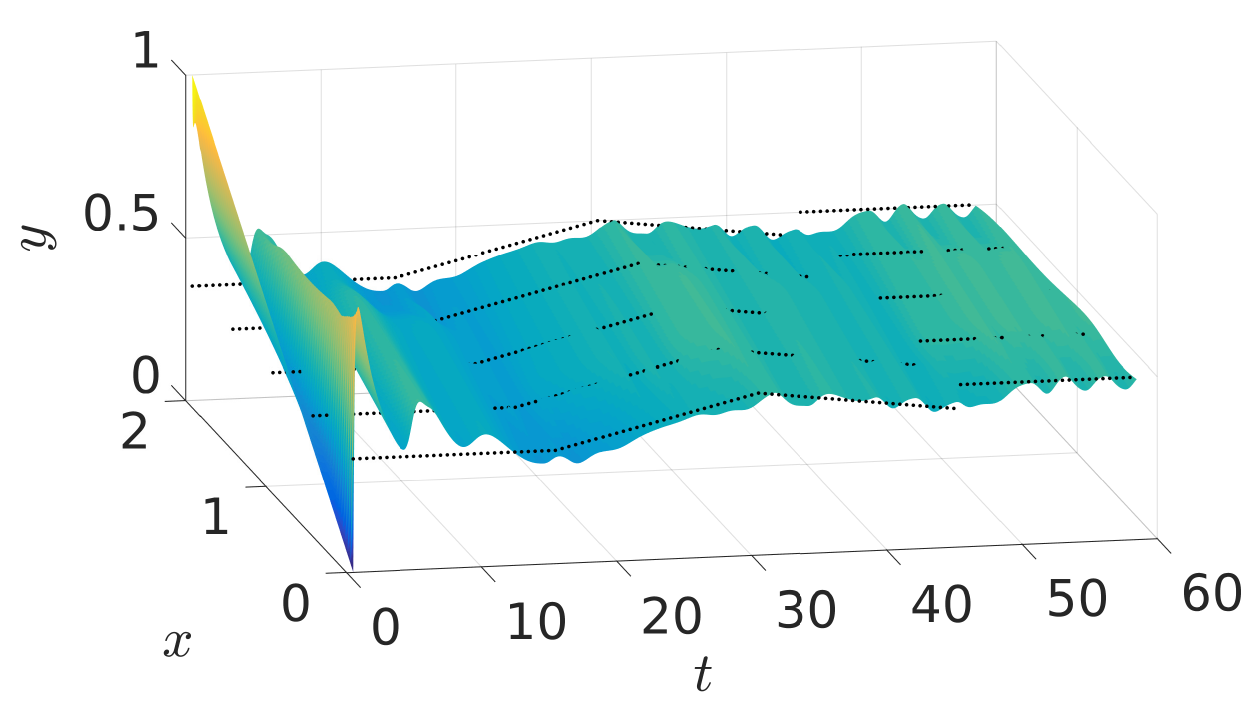} \\ \footnotesize{(d)}} \\[1ex]
	\parbox[b]{0.49\textwidth}{\centering \includegraphics[width=.4\textwidth]{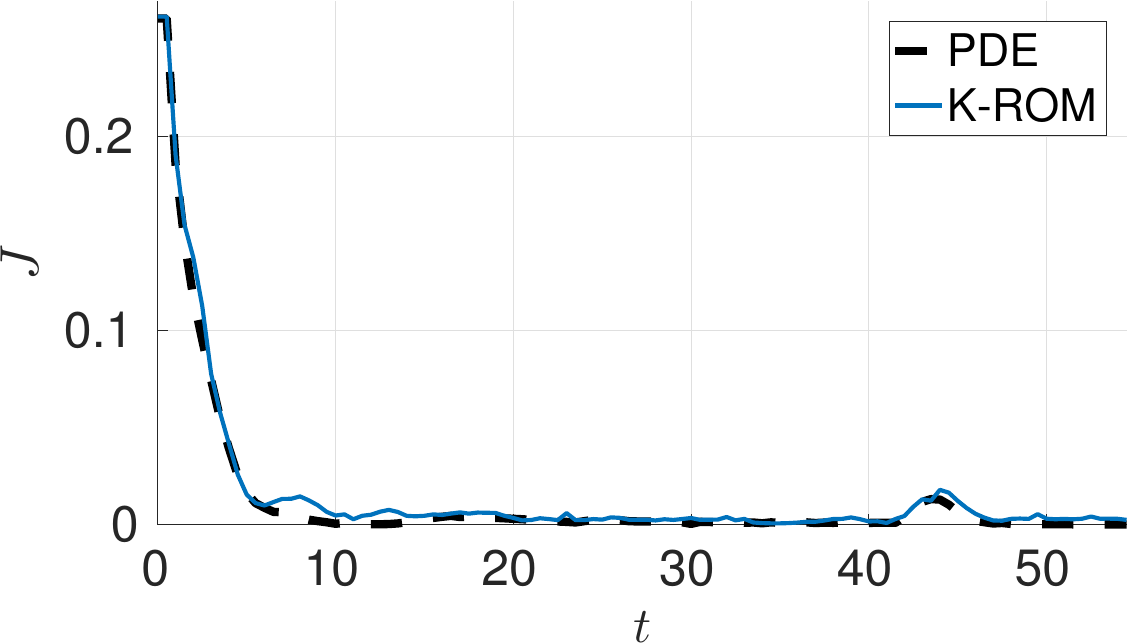} \\ \footnotesize{(e)}}
	\parbox[b]{0.49\textwidth}{\centering \includegraphics[width=.4\textwidth]{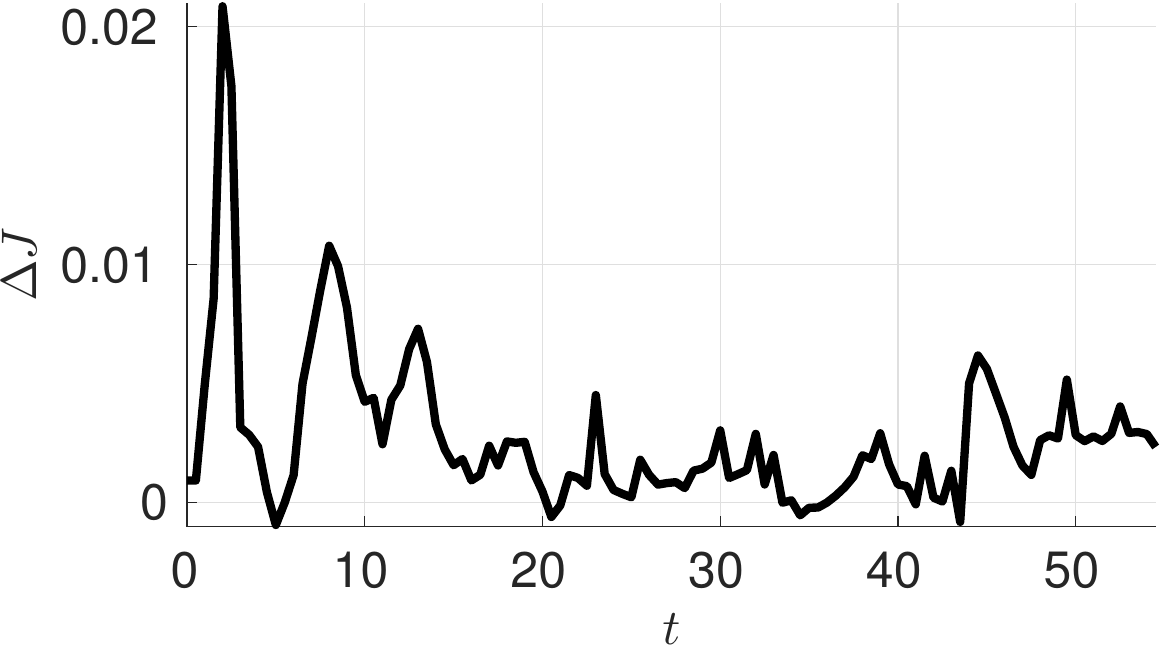} \\ \footnotesize{(f)}}
	\caption{(a) The shape functions $u^0$ to $u^2$. The positions where the state is observed by $f$ are shown in black. (b) Optimal right-hand side $j$ obtained by the K-ROM MPC approach. (c) Corresponding trajectories of the observations $z$. The reference trajectory (black dotted line) is identical for all four observables. (d) Full system state. (e) Value of the objective function during the MPC for the PDE-constrained and the K-ROM based problem (MPC setup is identical ($p=3$)). (f) Difference between the objective function values in (e).}
	\label{fig:Burgers}
\end{figure}

In contrast to the ODE case, we do not observe the entire state here but only certain points in space (the black dots in Figure~\ref{fig:Burgers}~(a)), i.e.,
\begin{equation*}
	z = f(y(\cdot,t)) = \left(y(0,t), y(0.5,t), y(1,t), y(1.5,t)\right)^{\top},
\end{equation*}
and we construct the K-ROM for these observations from data. Following Assumption~\ref{ass:equality_objectives_STO}, we formulate the tracking type objective function in terms of the observables only, which yields the following MPC problem:
\begin{equation*}
	\min_{u \in \{u^0,u^1,u^2\}^p} \sum_{i=s}^{s+p-1} \|z_i - z_{i}^{\mathsf{opt}} \|_2^2,
\end{equation*}
where $z_{i}^{\mathsf{opt}} = f(y^{\mathsf{opt}}(\cdot, t_i))$. Alternatively, we could directly observe the distance between the current and the reference state, i.e., $\int \|y(x,t) - y^{\mathsf{opt}}(x,t)\| \, dx$. This approach will be considered in the next example.

The results of the K-ROM-based MPC are shown in Figure~\ref{fig:Burgers}. We see that by using a low-dimensional linear model for the observations, we are able to control the PDE very accurately. This is even more remarkable since we have severely restricted the control space to only three inputs. Furthermore, the quality of the solution is almost equal to the full problem, cf.~Figure~\ref{fig:Burgers} (e) and (f).

\textbf{2D Navier--Stokes equations.} As a final example, we consider the flow of a fluid around a cylinder described by the 2D incompressible Navier--Stokes equations at a Reynolds number of $Re = 100$ (see Figure~\ref{fig:vonKarman}~(a) for the problem setup):
\begin{align*}
	\dot{y}(x,t) + y(x,t) \cdot \nabla y(x,t) &= \nabla p(x,t) + \frac{1}{Re} \Delta y(x,t), \\
	\nabla \cdot y(x,t) &= 0, \\
	y(x,0) &= y^0(x).
\end{align*}
The system is controlled via rotation of the cylinder, i.e., $u(t)$ is the angular velocity. The uncontrolled system possesses a periodic solution, the well-known \emph{von K\'{a}rm\'{a}n vortex street}.
\begin{figure}[t]
	\centering
	\parbox[b]{0.49\textwidth}{\centering \includegraphics[width=.4\textwidth]{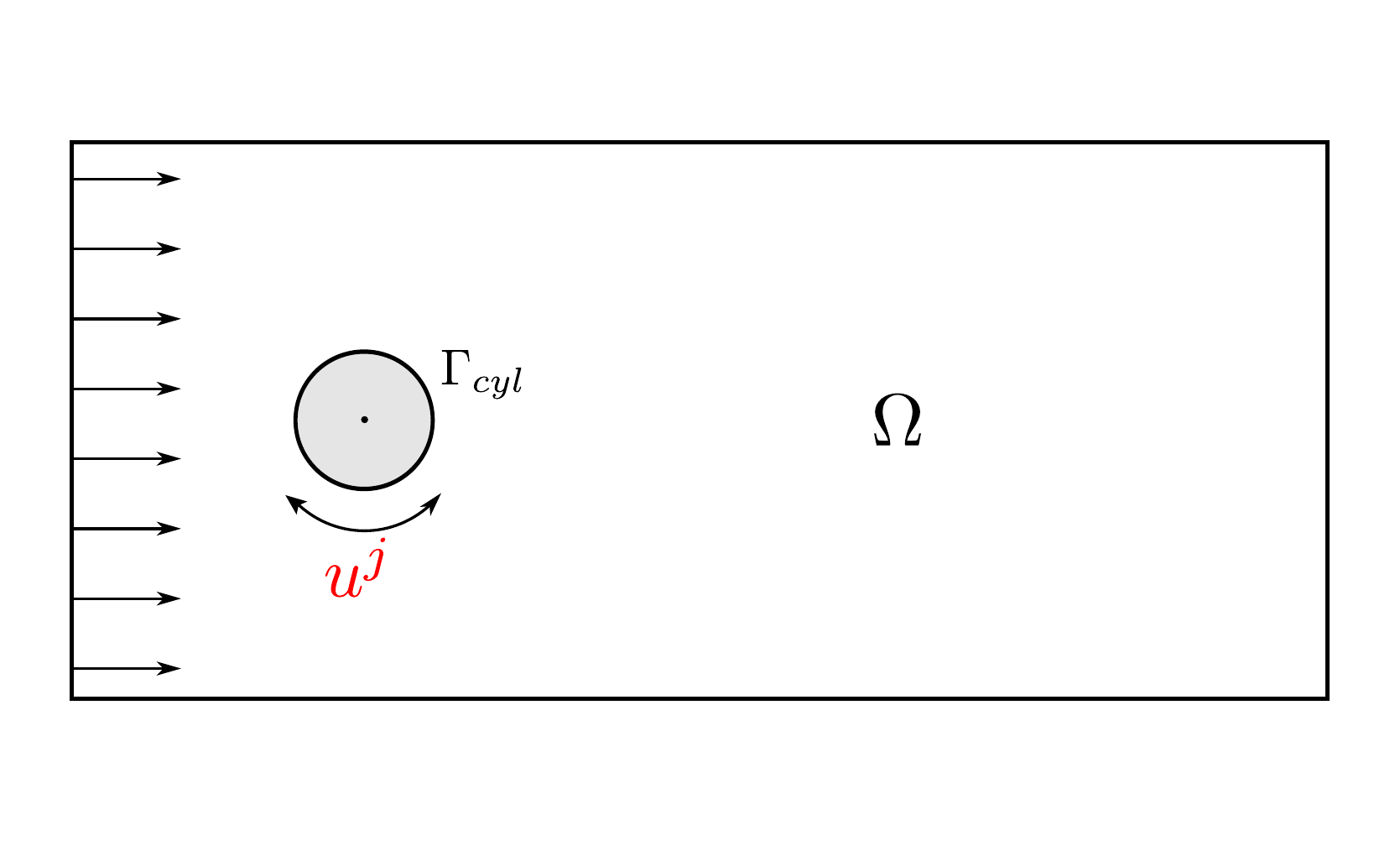} \\ \footnotesize{(a)}}
	\parbox[b]{0.49\textwidth}{\centering \includegraphics[width=.4\textwidth]{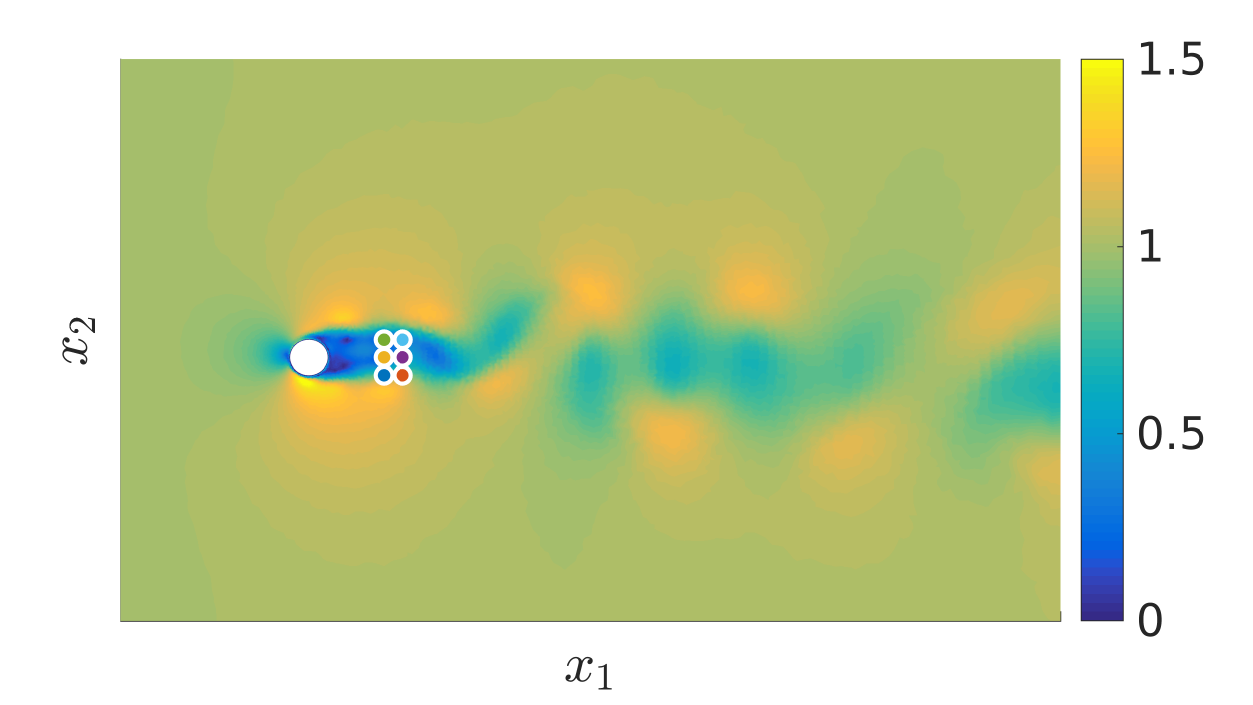} \\ \footnotesize{(b)}} \\[1ex]
	\parbox[b]{0.49\textwidth}{\centering \includegraphics[width=.4\textwidth]{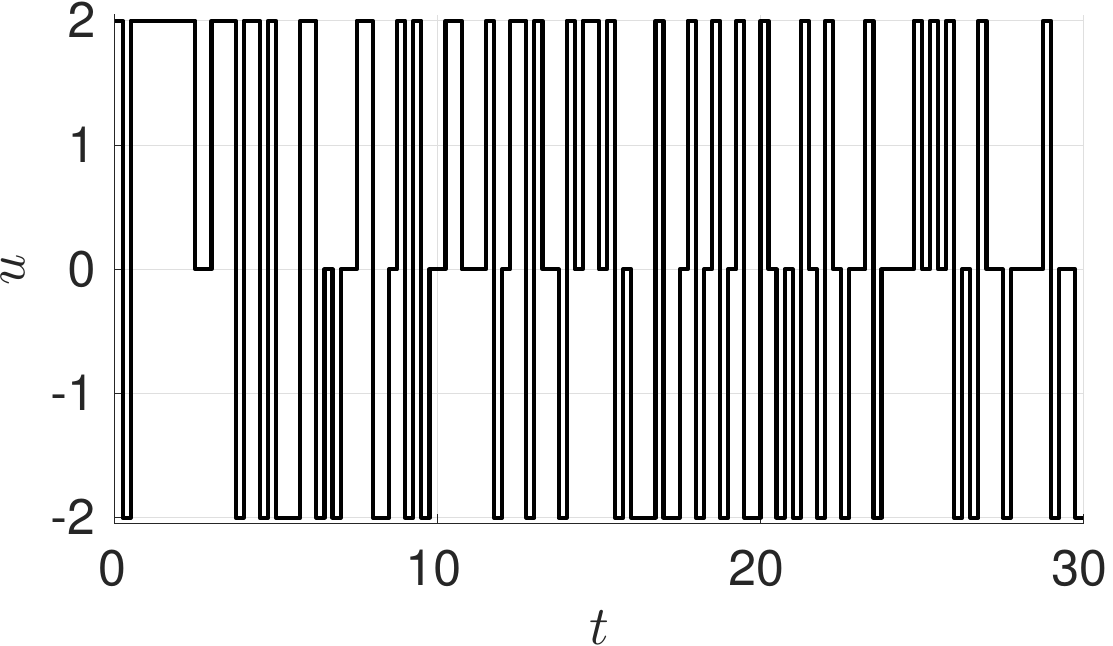} \\ \footnotesize{(c)}}
	\parbox[b]{0.49\textwidth}{\centering \includegraphics[width=.4\textwidth]{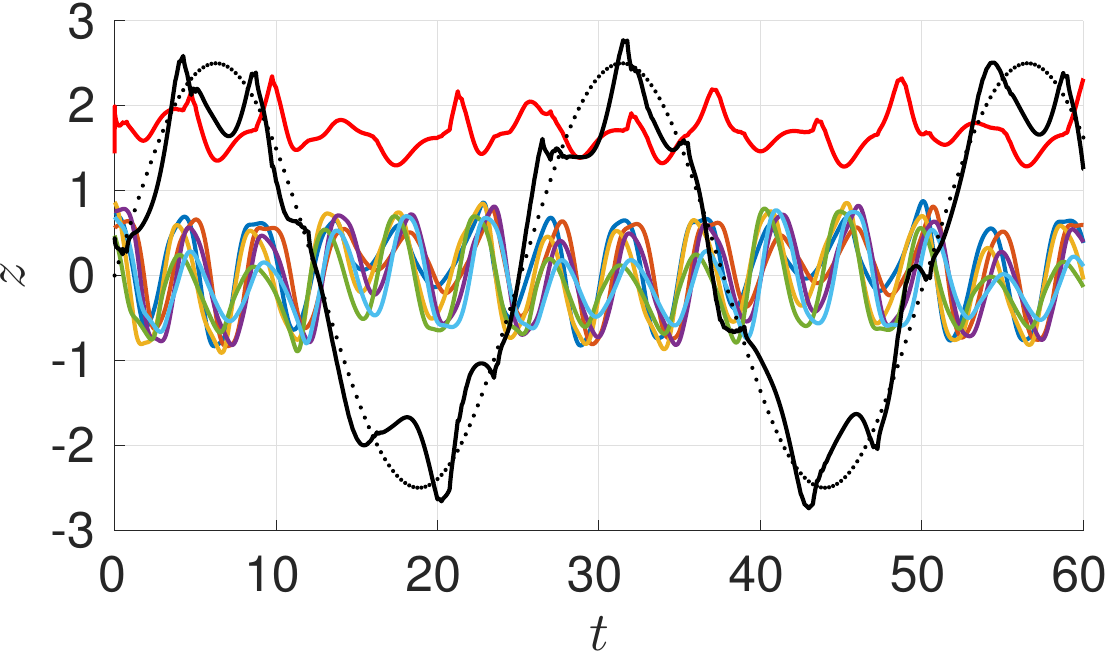} \\ \footnotesize{(d)}}
	\caption{(a) Problem setup. The system is controlled via rotation of the cylinder $\Gamma_{cyl}$. (b) Snapshot of the full system simulation for the sequence shown in (c), the sensor positions for the observations are marked by the colored dots. (c) Optimal control inputs obtained by the K-ROM MPC approach with $p=5$. (d) Corresponding observations $f((y(\cdot,t))$. The reference trajectory for the lift coefficient (black line) is marked by black dots, the drag is shown in red and the six horizontal velocities are colored according to the sensors shown in (b).}
	\label{fig:vonKarman}
\end{figure}

We now follow the same procedure as in the previous example. Instead of observing the full state, we observe the lift $C_l$ and the drag $C_d$ of the cylinder. Additionally, we observe the vertical velocity at six different positions $(x_1,\ldots,x_6)$ in the cylinder wake (see Figure~\ref{fig:vonKarman}~(b)):
\begin{equation*}
	z = \left(C_l(t), C_d(t), y_2(x_1,t), \ldots, y_2(x_6,t)\right)^{\top}.
\end{equation*}
The goal is to control the lift by rotating the cylinder. We transform the non-autonomous system into three autonomous ones with constant cylinder rotations $u^0 = 0$, $u^1 = 2$, $u^2 = -2$ and approximate the three corresponding Koopman operators. Since the lift coefficient is one of the observables, we simply have to track the corresponding entry of $z$ in the MPC problem:
\begin{align*}
	\min_{u \in \{u^0,u^1,u^2\}^p} \sum_{i=s}^{s+p-1} \left(z_{i,1} - z_{i}^{\mathsf{opt}}\right)^2.
\end{align*}
The switching sequence obtained by the K-ROM-based MPC algorithm is shown in Figure~\ref{fig:vonKarman}~(c), the corresponding dynamics of the K-ROM in Figure~\ref{fig:vonKarman}~(d). We see that the algorithm can successfully track the desired lift trajectory (shown as black dots) using only three different control inputs and a linear reduced model. We also observe some divergence, in particular for large reference values. This is due to the fact that the system cannot follow the prescribed desired state using the box constraints $[u^l,u^u] = [-2,2]$. An additional explanation for deviations may be that the data was not rich enough to achieve good performance for all three Koopman operator approximations, and predictions are not sufficiently accurate which results in the selection of the incorrect control input. This will be the subject of further research. In the first case, one could adapt the control bounds as well as add a few more controls. However, this will lead to an exponential increase in the number of possible solutions for the MPC problem such that special care has to be taken to maintain real-time applicability. In the second case, one could adopt ideas from \cite{HWR14}: during system operation, additional data is collected and then used to regularly update the Koopman operator approximations.

\section{Conclusion}
\label{sec:Conclusion}

We have presented a framework for open- and closed-loop control using Koopman operator-based reduced order models. By transforming the non-autonomous control system into a (small) number of autonomous systems with fixed control inputs, the control problem is transformed into a switching time problem. The approach enables us to control infinite-dimensional nonlinear systems using finite-dimensional, linear surrogate models. Using a recent convergence result for EDMD, we can prove optimality of the obtained solution. The numerical results show excellent performance, both considering the accuracy as well as the computing time. It will therefore be of great interest to further explore the limits and possibilities of this approach.

Further directions of research are stability properties of the K-ROM-based MPC method and a comparison in performance (e.g., the choice of $p$). Due to the identity of the optimal solutions, we expect that stability results can be adopted (for the infinite data limit). As already mentioned, it will be interesting to abandon the separation into offline and online phase and investigate the influence of regular updates using streaming data \cite{HWR14}. In situations where more multiple control inputs are required, it might become challenging to maintain real-time applicability. In this situation, relaxation techniques need to be exploited. It would also be of interest to further study the influence of the assumptions on the Koopman operator and whether convergence can be improved by choosing basis functions tailored to the system dynamics.

\noindent
\textbf{Acknowledgements:} We would like to thank the anonymous reviewers for their very helpful comments regarding the convergence analysis. This research was partially funded by the DFG Priority Programme 1962 ``Non-smooth and Complementarity-based Distributed Parameter Systems''. 

\bibliographystyle{alpha}
\bibliography{Bibliography}
\end{document}